\documentclass[11pt, leqno, english]{article}

\usepackage{babel,latexsym,graphicx,amsmath,amsfonts,amsthm}

\def\sign{\mathop{\mathrm{sign}}}
\def\wtilde{\widetilde}

\newtheoremstyle{localthm}
	{10pt} 
	{5pt} 
	{\sl} 
	{} 
	{\bf} 
	{{\rm.}} 
	{.7em} 
	{} 

\theoremstyle{localthm}
\newtheorem{Theorem}{Theorem}

\newtheorem{Lemma}{Lemma}

\newtheoremstyle{localrem}
	{10pt} 
	{5pt} 
	{\rm} 
	{} 
	{\bf} 
	{{\rm.}} 
	{.7em} 
	{} 

\theoremstyle{localrem}

\def\R{\mathbb{R}}

\begin{document}
\addtolength{\baselineskip}{0.05\baselineskip}

\title{Bounding Standard Gaussian Tail Probabilities}
\author{Lutz D{\"u}mbgen\\
	University of Bern}
\date{December 2010, revised December 2026}

\maketitle

\begin{abstract}
We review various inequalities for Mills' ratio $(1 - \Phi)/\phi$, where $\phi$ and $\Phi$ denote the standard Gaussian density and distribution function, respectively. Elementary considerations involving finite continued fractions lead to a general approximation scheme which implies and improves several known bounds.
\end{abstract}

\section{Introduction}

Explicit formulae for the distribution function $\Phi$ of the standard Gaussian distribution are unknown, apart from various expansions, e.g.\ the series expansion
\[
	\Phi(x) \ = \ \frac{1}{2}
	+ \frac{1}{\sqrt{2\pi}} \sum_{k=0}^\infty \frac{(-1)^k x^{2k+1}}{2^k k! (2k+1)}
\]
for real $x$, or Laplace's (1805) continued fractions expansion
\begin{equation}
\label{continued fractions}
	1 - \Phi(x) \, = \, \Phi(-x) \ = \ \frac{\phi(x)}
		{x + \displaystyle\frac{1}{x + 
				\displaystyle\frac{2}{x +
					\displaystyle\frac{3}{\ \ddots}}}}
\end{equation}
for $x > 0$. Here $\phi = \Phi'$ denotes the standard Gaussian density. We refer to Abramowitz and Stegun (1972, Chapter 7) for these and numerous further results about the function $\Phi$. The expansion \eqref{continued fractions} indicates that there should be good approximations or bounds on $1 - \Phi(x)$ of the form
\[
	\Phi(-x) \ \approx \ \frac{\phi(x)}{h(x)} ,
	\quad x > 0 ,
\]
with $h : (0,\infty) \to (0,\infty)$ being a relatively simple function. For instance, let $h_0(x) := x$, $h_1(x) := x + 1/x$ and, for $k \ge 2$,
\[
	h_k(x) \ := \ x + \frac{1}{x + 
				\displaystyle\frac{2}{\begin{array}{c}\ddots\\\strut\end{array} x +
					\displaystyle\frac{k}{x}}} .
\]
Then it is known that for integers $k \ge 0$ and arbitrary $x > 0$,
\begin{equation}
\label{finite continued fractions}
	\left. \begin{array}{c}
		\displaystyle
		\frac{\phi(x)}{h_0(x)} \, > \, \frac{\phi(x)}{h_2(x)} \, > \, \frac{\phi(x)}{h_4(x)} \, \cdots \ > \\[2.5ex]
		\displaystyle
		\frac{\phi(x)}{h_1(x)} \, < \, \frac{\phi(x)}{h_3(x)} \, < \, \frac{\phi(x)}{h_5(x)} \, \cdots \ <
	\end{array} \right\} \ \Phi(-x)  .
\end{equation}
In particular,
\begin{equation}
\label{Gordon}
	\frac{\phi(x)}{x + 1/x} \ < \ \Phi(-x) \ < \ \frac{\phi(x)}{x}
	\quad\text{for} \ x > 0 ,
\end{equation}
which was established by Gordon (1941). These special functions $h_k$ and their properties have been investigated by numerous authors; we only refer to Shenton (1954), Pinelis (2002), Baricz (2008) and the references therein.

While the previous bounds are only useful for $x$ bounded away from zero, some authors provided inequalities on the whole interval $[0,\infty)$ or even larger sets. Komatu (1955) showed that
\begin{equation}
\label{Komatu}
	\frac{2 \phi(x)}{\sqrt{4 + x^2} + x}
	\ < \ \Phi(-x) \ < \ \frac{2 \phi(x)}{\sqrt{2 + x^2} + x}
	\quad\text{for} \ x \ge 0 ;
\end{equation}
see also Ito and McKean (1974). In fact, the lower bound is due to Birnbaum (1942) who formulated it equivalently as $\Phi(-x) > \bigl( \sqrt{4 + x^2} - x \bigr) \phi(x) / 2$ for $x \ge 0$. Pollak (1956) refined Komatu's upper bound as follows:
\begin{equation}
\label{Pollak}
	\Phi(-x) \ < \ \frac{2 \phi(x)}{\sqrt{8/\pi + x^2} + x}
	\quad\text{for} \ x > 0 .
\end{equation}
An alternative upper bound, due to Sampford (1953) and rediscovered by Szarek and Werner (1999), reads
\begin{equation}
\label{Sampford}
	\Phi(-x) \ < \ \frac{4 \phi(x)}{\sqrt{8 + x^2} + 3 x}
	\quad\text{for} \ x \ge 0 .
\end{equation}
Shenton (1954) and Kouba (2006) generalized Komatu's bound \eqref{Komatu} and Sampford's bound \eqref{Sampford} substantially. Here is a reformulation of their bounds with continued fractions: For integers $k \ge 0$ and $j = 1,2$ define $h_{k,j} : [0,\infty) \to (0,\infty)$ via $h_{0,j}(x) := g_{0,j}(x)$, $h_{1,j}(x) = x + 1/g_{1,j}(x)$ and
\[
	h_{k,j}(x) \ = \ x + \frac{1}{x + 
				\displaystyle\frac{2}{\begin{array}{c}\ddots\\\strut\end{array} x +
					\displaystyle\frac{k}{g_{k,j}(x)}}}
\]
if $k \ge 2$, where $g_{k,j}(x) := \sqrt{k + j/2 + (x/2)^2} + x/2$. Then for all $x \ge 0$,
\begin{equation}
\label{Shenton}
	\frac{\phi(x)}{h_{k,j}(x)} \
	\begin{cases}
		< \ \Phi(-x) & \text{if} \ k + j \ \text{is even} , \\
		> \ \Phi(-x) & \text{if} \ k + j \ \text{is odd} .
	\end{cases}
\end{equation}

In the present manuscript we present all these bounds in a common framework and propose improvements. In Section~\ref{First Steps} we consider the derivative of the approximation error
\[
	\Delta(x) \ := \frac{\phi(x)}{h(x)} - \Phi(-x)
\]
for an arbitrary smooth function $h : (0,\infty) \to (0,\infty)$ such that $\liminf_{x\to\infty} h(x) > 0$. In addition, we consider $h(x) = x + 1/g(x)$ with another smooth function $g : (0,\infty) \to (0,\infty)$. Analyzing the derivative $\Delta'$ for some special choices of $h$ and $g$ leads to various bounds, including \eqref{Komatu}, \eqref{Pollak}, \eqref{Sampford} and refinements thereof.

In Section~\ref{Continued Fractions} we consider approximations $\phi(x)/h_k(x)$ of $\Phi(-x)$, where $h_0(x) = g_0(x)$, $h_1(x) = x + 1/g_1(x)$ and
\[
	h_k(x) \ = \ x + \frac{1}{x + 
				\displaystyle\frac{2}{\begin{array}{c}\ddots\\\strut\end{array} x +
					\displaystyle\frac{k}{g_k(x)}}}
\]
for $k \ge 2$. Here $g_k : (0,\infty) \to (0,\infty)$ is a smooth function. It turns out that under general conditions on $g_k$, this yields upper and lower bounds $\phi/h_k$ for $1 - \Phi$, including \eqref{finite continued fractions},\eqref{Shenton} as well as new refinements.

\section{First Steps}
\label{First Steps}

\subsection{Some first bounds of type $\phi/h$}

Let $h : (0,\infty) \to (0,\infty)$ be a differentiable function with $\liminf_{x \to \infty} h(x) > 0$. Then the approximation error $\Delta(x) = \phi(x)/h(x) - \Phi(-x)$ satisfies
\[
	\Delta(\infty) \, := \, \lim_{x \to \infty} \, \Delta(x) \ = \ 0 ,
\]
and
\[
	\Delta'(x)
	\ = \ \frac{\phi(x)}{h(x)^2} \, H(x) ,
	\quad
	H(x) \ := \ h(x)^2 - x h(x) - h'(x) .
\]
In particular,
\[
	\sign(\Delta'(x)) \ = \ \sign(H(x)) .
\]
This provides us with the following two recipes to find bounds for $\Phi(-x)$ which will be used repeatedly.

\paragraph{\sl Recipe~1($h$).} If $H > 0$ on $(0,\infty)$, then $\Delta < 0$ on $[0,\infty)$ with $\Delta(0) = \lim_{x\to 0} \Delta(x) \in [-\infty,0)$. This follows from $\Delta' > 0$ on $(0,\infty)$ and $\Delta(\infty) = 0$.\\[0.25ex]
Analogously, if $H < 0$ on $(0,\infty)$, then $\Delta > 0$ on $[0,\infty)$.

\paragraph{\sl Recipe~2($h$).} Suppose that $h$ is defined and differentiable on $[0,\infty)$ with $h(0) = \sqrt{2/\pi}$, so $\Delta(0) = 0$.\\[0.25ex]
Suppose that for some $x_o > 0$, $H > 0$ on $(0,x_o)$ and $H < 0$ on $(x_o,\infty)$. Then $\Delta > 0$ on $(0,\infty)$. Indeed, the assumption on $H$ implies that $\Delta$ is strictly increasing on $[0,x_o]$ and strictly decreasing on $[x_o,\infty)$. Together with $\Delta(0) = 0 = \Delta(\infty)$, this implies that $\Delta > 0$ on $(0,\infty)$.\\[0.25ex]
Analogously, if for some $x_o > 0$, $H < 0$ on $(0,x_o)$ and $H > 0$ on $(x_o,\infty)$, then $\Delta < 0$ on $(0,\infty)$.

One could generalize Recipe~2($h$) to cover situations with $h(0) \ne \sqrt{2/\pi}$, but for our purposes, the current version is sufficient.

\paragraph{Affine functions $h$.}
Consider first $h(x) = x$. Then $H \equiv - 1$, so $\Delta > 0$ on $(0,\infty)$ by Recipe~1($h$). This yields Gordon's upper bound in \eqref{Gordon}.

Alternatively, for some $\lambda \in [0,1]$, let
\[
	h(x) \ = \ \sqrt{2/\pi} + \lambda x , \quad x \ge 0 .
\]
Then, $h(0) = \sqrt{2/\pi}$ and
\[
	H(x) \
	= \ 2/\pi - \lambda + (2\lambda - 1) \sqrt{2/\pi} \, x - \lambda(1 - \lambda) x^2 .
\]

If $\lambda = 1$, then $H(x) = 2/\pi - 1 + \sqrt{2/\pi} \, x = \sqrt{2/\pi}(x - x_o)$ with $x_o := \sqrt{\pi/2} - \sqrt{2/\pi}$. Thus, $\Delta' < 0$ on $[0,x_o)$ and $\Delta' > 0$ on $(x_o,\infty)$, so Recipe~2($h$) shows that $\Delta < 0$ on $(0,\infty)$.

If $\lambda = 2/\pi$, then $H(x) = (2/\pi)(1 - 2/\pi) x (x_o - x)$ with $x_o := (4/\pi - 1) \big/ \bigl[ \sqrt{2/\pi} (1 - 2/\pi) \bigr]$, so $\Delta' > 0$ on $(0,x_o)$ and $\Delta' < 0$ on $(x_o,\infty)$. Thus, by Recipe~2($h$), $\Delta > 0$ on $(0,\infty)$.

All in all, these considerations show that
\begin{equation}
\label{Simple0}
	\frac{\phi(x)}{\sqrt{2/\pi} + x} \
	< \ \Phi(-x) \
	< \ \frac{\phi(x)}{\max \bigl\{ \sqrt{2/\pi} + (2/\pi) x, x \bigr\}}
	\quad \text{for} \ x > 0 .
\end{equation}

\paragraph{A special function type.}
In view of the bounds \eqref{Komatu}, \eqref{Pollak} and \eqref{Sampford}, we try
\[
	h(x) \ := \ h_o(x) + (1 - \tau) x, \quad h_o(x) \ := \ \sqrt{c + \tau^2 x^2} ,
\]
for some constants $c > 0$ and $\tau \in [0,1]$ to be specified later. Then $\Delta$ is well-defined and differentiable on $[0,\infty)$, and one verifies easily that
\begin{align*}
	H(x) \
	&= \ c - 1 + \tau + (1 - 2\tau) x (h_o(x) - \tau x) - \frac{\tau^2 x}{h_o(x)} \\
	&= \ c - 1 + \tau + \frac{(1 - 2\tau) c x}{h_o(x) + \tau x} - \frac{\tau^2 x}{h_o(x)} .
\end{align*}

\paragraph{\sl Special case 1: $\tau = 1/2$.}
Here
\[
	H(x) \ = \ c - 1/2 - \frac{x}{2 \sqrt{4c + x^2}}
\]
is strictly decreasing in $x \in [0,\infty)$ with $H(0) = c - 1/2$ and $H(x) \to c - 1$ as $x \to \infty$.

If $c = 1$, then $H > 0$ on $[0,\infty)$, so Recipe~1($h$) shows that $\Delta < 0$ on $[0,\infty)$. This yields Birnbaum's and Komatu's lower bound in \eqref{Komatu}, noting that $\sqrt{1 + x^2/4} + x/2 = \bigl( \sqrt{4 + x^2} + x \bigr) / 2$.

If $c = 1/2$, then $H(x) < 0$ for all $x > 0$, so Tecipe~1$(h)$ shows that $\Delta > 0$ on $[0,\infty)$. This implies Komatu's upper bound in \eqref{Komatu}, noting that $\sqrt{1/2 + x^2/4} + x/2 = \bigl( \sqrt{2 + x^2} + x \bigr) / 2$.

But the latter choice of $c$ for an upper bound is suboptimal. For if $c = 2/\pi$, then $h(0) = \sqrt{2/\pi}$ and $H(0) = 2/\pi - 1/2 > 0$ and $\lim_{x\to\infty} H(x) = 2/\pi - 1 < 0$. Consequently, we can apply Recipe~2($h$) and deduce that $\Delta > 0$ on $(0,\infty)$. This leads to Pollak's upper bound \eqref{Pollak}, noting that $\sqrt{2/\pi + x^2/4} + x/2 = \bigl( \sqrt{8/\pi + x^2} + x \bigr) / 2$.

\paragraph{\sl Special case 2: $c = 2/\pi$ and $\tau = 1 - c$.}
Inspired by Pollak's upper bound \eqref{Pollak}, we keep $c = 2/\pi$, so $\Delta(0) = 0$, and the choice $\tau = 1 - c$ yields $H(0) = 0$, so $\Delta'(0) = 0$ too. Now,
\begin{align*}
	H(x) \
	&= \ \frac{(2c - 1) c x}{h_o(x) + \tau x} - \frac{\tau^2 x}{h_o(x)} \\
	&= \ \frac{x}{(h_o(x) + \tau x) h_o(x)}
		\bigl( (c^2 + c - 1) h_o(x) - \tau^3 x \bigr) \\
	&= \ \frac{(c^2 + c - 1) x^2}{(h_o(x) + \tau x) h_o(x)}
		\Bigl( \sqrt{\frac{c}{x^2} + \tau^2} - \frac{\tau^3}{c^2 + c - 1} \Bigr) .
\end{align*}
Note that $c^2 + c - 1 > (2/3.15)^2 + 2/3.15 - 1 > 0$. Moreover, the factor
\[
	H^o(x) \ := \ \sqrt{\frac{c}{x^2} + \tau^2} - \frac{\tau^3}{c^2 + c - 1}
\]
is strictly decreasing in $x > 0$ with $\lim_{x\to 0} H^o(x) = \infty$ and
\[
	\lim_{x\to\infty} H^o(x) \ = \ \tau - \frac{\tau^3}{c^2 + c - 1} \
	= \ \frac{2\tau (3/\pi - 1)}{c^2 + c - 1} \ < \ 0 .  
\]
Consequently, Recipe~2($h$) leads to the new upper bound
\begin{equation}
\label{new.upper.0}
	\Phi(-x) \ < \ \frac{\phi(x)}{\sqrt{2/\pi + (1 - 2/\pi)^2 x^2} + (2/\pi) x}
	\quad \text{for} \ x > 0 .
\end{equation}
Note that this bound is strictly better than the upper bound in \eqref{Simple0}, because
\[
	\sqrt{2/\pi + (1 - 2/\pi)^2 x^2} + (2/\pi) x \ > \ \max \bigl\{ \sqrt{2/\pi} + (2/\pi)x, x \bigr\} .
\]
It is also strictly better than Pollak's bound \eqref{Pollak}, because of $1 - 2/\pi < 1/2$ and the following general fact: For any $x > 0$, $c > 0$ and $\tau \in \R$,
\begin{equation}
\label{monotonicity in tau}
	\frac{\mathrm{d}}{\mathrm{d}\tau} \bigl( \sqrt{c + \tau^2 x^2} + (1 - \tau) x \bigr) \
	= \ \Bigl( \frac{\tau x}{\sqrt{c + \tau^2 x^2}} - 1 \Bigr) x \ < \ 0 .
\end{equation}

\subsection{Some refined bounds with $h(x) = x + 1/g(x)$}
\label{subsec:refined.h.g}

In the previous subsection one might wonder whether other choices of $c > 0$ and $\tau \in (0,1)$ for $h(x) = \sqrt{c + \tau^2x^2} + (1 - \tau)x$ yield good bounds too. For instance, Sampford's bound \eqref{Sampford} uses
\[
	h(x) \ = \ \bigl( \sqrt{8 + x^2} + 3x \bigr) / 4 \
	= \ \sqrt{1/2 + (1/4)^2 x^2} + (3/4) x ,
\]
i.e.\ $c = 1/2$ and $\tau = 1/4$. But note that in general,
\begin{align*}
	h(x) \, = \, \sqrt{c + \tau^2 x^2} - (1 - \tau) x \
	&= \ x + \sqrt{c + (\tau x)^2} - \tau x \\
	&= \ x + \frac{c}{\sqrt{c + (\tau x)^2} + \tau x} \\
	&= \ x + \frac{1}{g(x)}
\end{align*}
with $g(x) := \sqrt{1/c + (\tau x/c)^2} + \tau x/c$. Specifically, the function $h$ for Sampford's (1953) upper bound \eqref{Sampford} corresponds to $g(x) = \sqrt{2 + (x/2)^2} + x/2$. Thus we now consider functions $h$ on $(0,\infty)$ or $[0,\infty)$ given by
\[
	h(x) \ := \ x + \frac{1}{g(x)}
\]
for some smooth function $g > 0$ on $(0,\infty)$ or $[0,\infty)$. In any case, $\lim_{x\to\infty} h(x) = \infty$, so still $\Delta(\infty) = 0$. Furthermore, $h(0) = \sqrt{2/\pi}$ is equivalent to $g(0) = \sqrt{\pi/2}$, and one can easily verify that $H(x) = h(x)^2 - x h(x) - h'(x)$ equals
\[
	H(x) \ = \ \frac{-1}{g(x)^2} \, G(x)
	\quad\text{with}\quad  G(x) \ := \ g(x)^2 - x g(x) - 1 - g'(x) .
\]
Consequently, our two recipes for the function $h$ may be reformulated in terms of $g$:

\paragraph{\sl Recipe~1($g$).} If $G > 0$ on $(0,\infty)$, then $\Delta > 0$ on $[0,\infty)$. Analogously, if $G < 0$ on $(0,\infty)$, then $\Delta < 0$ on $[0,\infty)$.

\paragraph{\sl Recipe~2($g$).} Suppose that $g$ is defined and differentiable on $[0,\infty)$ with $g(0) = \sqrt{\pi/2}$, so $\Delta(0) = 0$.\\[0.25ex]
Suppose that for some $x_o > 0$, $G > 0$ on $(0,x_o)$ and $G < 0$ on $(x_o,\infty)$. Then $\Delta < 0$ on $(0,\infty)$.\\[0.25ex]
Analogously, if for some $x_o > 0$, $G < 0$ on $(0,x_o)$ and $G > 0$ on $(x_o,\infty)$, then $\Delta > 0$ on $(0,\infty)$.

\paragraph{Affine functions $g$.}
In case of $g(x) = x$, we obtain $G \equiv -2$, so $\Delta < 0$ on $(0,\infty)$ by Recipe~1($g$). This is Gordon's lower bound in \eqref{Gordon}.

Alternatively, let $\lambda > 0$ and
\[
	g(x) \ = \ \sqrt{\pi/2} + \lambda x , \quad x \ge 0 .
\]
In particular, $g(0) = \sqrt{\pi/2}$, and
\[
	G(x) \
	= \ \pi/2 - \lambda - 1 + (2\lambda - 1) \sqrt{\pi/2} \, x - \lambda(1 - \lambda) x^2 .
\]

If $\lambda = 1$, we obtain $G(x) = \pi/2 - 2 + \sqrt{\pi/2} \, x = \sqrt{\pi/2}(x - x_o)$ with the constant $x_o := (2 - \pi/2) / \sqrt{\pi/2} > 0$. Thus we can deduce via Recipe~2($g$) that $\Delta > 0$ on $(0,\infty)$.

If $\lambda = \pi/2 - 1 \in (1/2,1)$, then $G(x)$ is equal to $\lambda(1 - \lambda) x (x_o - x)$ with the constant $x_o := (2\lambda - 1) \sqrt{\pi/2} / [\lambda(1 - \lambda)] > 0$. Hence, Recipe~2($g$) implies that $\Delta < 0$ on $(0,\infty)$.

All in all, these considerations show that for $x > 0$,
\begin{equation}
\label{Simple1}
	\Phi(-x) \ \begin{cases}
		\displaystyle
		< \ \frac{\phi(x)}{x + 1 \big/ \bigl( \sqrt{\pi/2} + x \bigr)} , \\[2.5ex]
		\displaystyle
		> \ \frac{\phi(x)}{x + 1 \bigl/ \max \bigl\{ \sqrt{\pi/2} + (\pi/2 - 1) x, x \bigr\}} .
	\end{cases}
\end{equation}

\paragraph{A special function type.}
Again we consider
\[
	g(x) \ := \ g_o(x) + (1 - \tau) x , \quad
	g_o(x) \ := \ \sqrt{c + \tau^2 x^2} ,
\]
for some constants $c > 0$ and $\tau \in [0,1]$. This leads to
\[
	G(x) \
	= \ c - 2 + \tau + \frac{(1 - 2\tau) c x}{g_o(x) + \tau x} - \frac{\tau^2 x}{g_o(x)} .
\]

\paragraph{\sl Special case~1: $\tau = 1/2$.}
Here
\[
	G(x) \ = \ c - 3/2 - \frac{x}{2 \sqrt{4c + x^2}}
\]
is strictly decreasing in $x \ge 0$ with $G(0) = c - 3/2$ and $\lim_{x\to\infty} G(x) = c - 2$.

If $c = 2$, we see that $G > 0$ on $[0,\infty)$, so Recipe~1($g$) leads to the upper bound \eqref{Sampford} of Sampford.

If $c = \pi/2$, then $g(0) = \sqrt{\pi/2}$, $G(0) = (\pi - 3)/2 > 0$ and $\lim_{x\to\infty} G(x) = (\pi - 4)/2 < 0$. Consequently, Recipe~2($g$) leads to the new lower bound
\begin{equation}
\label{Pollak.lower}
	\Phi(-x) \ > \ \frac{\phi(x)}
		{x + 1 \bigl/ \bigl( \sqrt{\pi/2 + x^2/4} + x/2 \bigr)}
	\quad \text{for} \ x > 0 .
\end{equation}

\paragraph{\sl Special case~2: $c = \pi/2$ and $\tau = 2 - c$.}
With these choices for $c$ and $\tau$, we obtain $g(0) = \sqrt{\pi/2}$, and
\begin{align*}
	G(x) \
	&= \ \frac{(2c - 3) c x}{g_o(x) + \tau x} - \frac{\tau^2 x}{g_o(x)} \\
	&= \ \frac{x}{(g_o(x) + \tau x) g_o(x)}
		\bigl( (c^2 + c - 4) g_o(x) - \tau^3 x \bigr) \\
	&= \ \frac{(c^2 + c - 4) x^2}{(g_o(x) + \tau x) g_o(x)}
		\Bigl( \sqrt{\frac{c}{x^2} + \tau^2} - \frac{\tau^3}{c^2 + c - 4} \Bigr) .
\end{align*}
Note that $c^2 + c - 4 > (1.57)^2 + 1.57 - 4 > 0$. Moreover,
\[
	G^o(x) \ := \ \sqrt{\frac{c}{x^2} + \tau^2} - \frac{\tau^3}{c^2 + c - 4}
\]
is strictly decreasing in $x > 0$ with $\lim_{x\to 0} G^o(x) = \infty$ and
\[
	\lim_{x\to\infty} G^o(x) \ = \ \tau - \frac{\tau^3}{c^2 + c - 4} \
	= \ \frac{2.5 \tau (\pi - 3.2)}{c^2 + c - 1} \ < \ 0 .  
\]
Consequently, there exists an $x_o > 0$ such that $G > 0$ on $(0,x_o)$ and $G < 0$ on $(x_o,\infty)$. With Recipe~2($g$), we obtain the new lower bound
\begin{equation}
\label{new.lower.1}
	\Phi(-x) \ > \ \frac{\phi(x)}
		{x + 1 \bigl/ \bigl( \sqrt{\pi/2 + (2 - \pi/2)^2 x^2} + (\pi/2 - 1) x \bigr)}
	\quad \text{for} \ x > 0 .
\end{equation}
Since $2 - \pi/2 = (4 - \pi)/2 < 1/2$, one can deduce from \eqref{monotonicity in tau} that this lower bound is strictly better than \eqref{Pollak.lower}. It is also better than the lower bounds in \eqref{Simple0} and \eqref{Simple1}, because
\[
	\sqrt{\pi/2 + (2 - \pi/2)^2 x^2} + (\pi/2 - 1) x \
		> \ \max \bigl\{ \sqrt{\pi/2}, x, \sqrt{\pi/2} + (\pi/2 - 1) x \bigr\}
\]
for all $x > 0$.

\subsection{The best bounds so far}

We have verified already that \eqref{new.lower.1} is the strongest lower bound so far. Among the upper bounds, Sampford's bound \eqref{Sampford}, the new bound \eqref{new.upper.0} and the ad hoc upper bound in \eqref{Simple1} outperform all others, while none of them dominates the others everywhere. Indeed, as explained at the beginning of Section~\ref{subsec:refined.h.g}, the approximations $\phi(x)/h(x)$ are given by $h(x) = x + 1/g(x)$ with
\[
	g(x) \ = \ \begin{cases}
		\sqrt{2 + (x/2)^2} + x/2
			& \text{for \eqref{Sampford}} , \\
		\sqrt{\pi/2 + (\pi/2 - 1)^2 x^2} + (\pi/2 - 1) x
			& \text{for \eqref{new.upper.0}} , \\
		\sqrt{\pi/2} + x
			& \text{for \eqref{Simple1}} .
	\end{cases}
\]
On the one hand, the upper bound in \eqref{Simple1} is better than \eqref{Sampford} if and only if $\sqrt{\pi/2} + x$ is smaller than $\sqrt{2 + (x/2)^2} + x/2$, and this is equivalent to $x < x_1 := (2 - \pi/2)/\sqrt{\pi/2} \approx 0.3425$. On the other hand, \eqref{new.upper.0} is at least as good as \eqref{Simple1} if and only if $f(x) := \sqrt{\pi/2 + (\pi/2 - 1)^2 x^2} + (\pi/2 - 2)x \le \sqrt{\pi/2}$. But $f(0) = \sqrt{\pi/2}$, and
\[
	f'(x) \ = \ \frac{(\pi/2 - 1)^2 x}{\sqrt{\pi/2 + (\pi/2 - 1)^2 x^2}} - (2 - \pi/2)
\]
is strictly increasing in $x \ge 0$ with $f'(0) < 0$. Hence \eqref{new.upper.0} is at least as good as \eqref{Simple1} if $f'(x) \le 0$. For instance, if $x \le 2$, then $f'(x) \le f'(2) < 0$. Consequently, the upper bound in \eqref{Simple1} is outperformed by the pointwise minimum of \eqref{Sampford} and \eqref{new.upper.0}.
 
Figure~\ref{fig:B1} shows the resulting approximation errors $\Delta(x) = \phi(x)/h(x) - \Phi(-x)$ as well as the relative errors $\Delta(x)/\Phi(-x)$. Note that the upper bound \eqref{new.upper.0} has a relative error in $[0,0.01)$, and the lower bound \eqref{new.lower.1} has a relative error in $(-0.002,0]$.

\begin{figure}[h]
\includegraphics[width=0.49\textwidth]{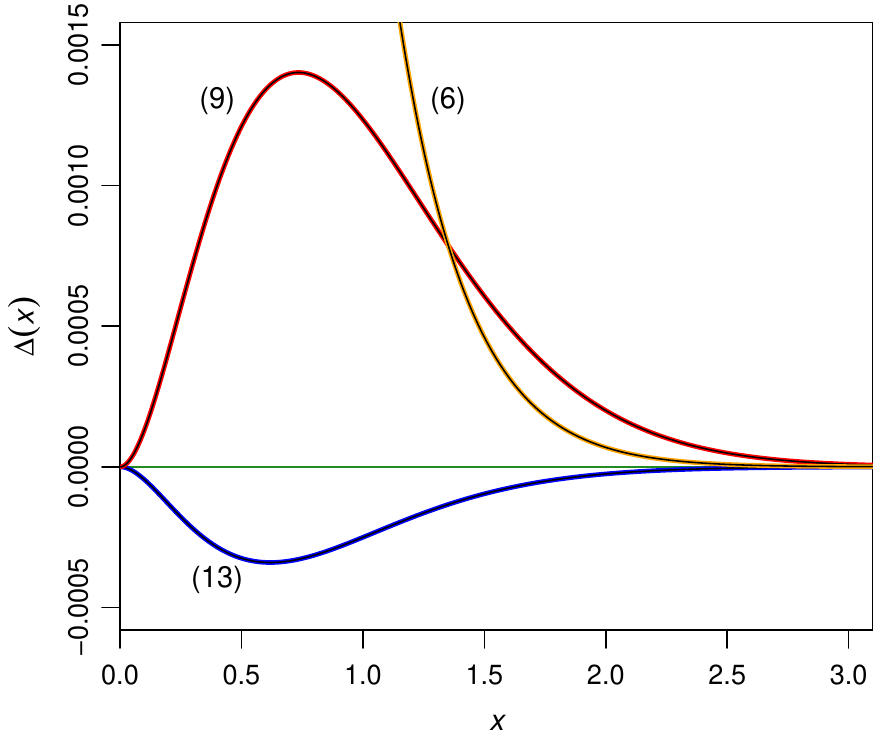}
\hfill
\includegraphics[width=0.49\textwidth]{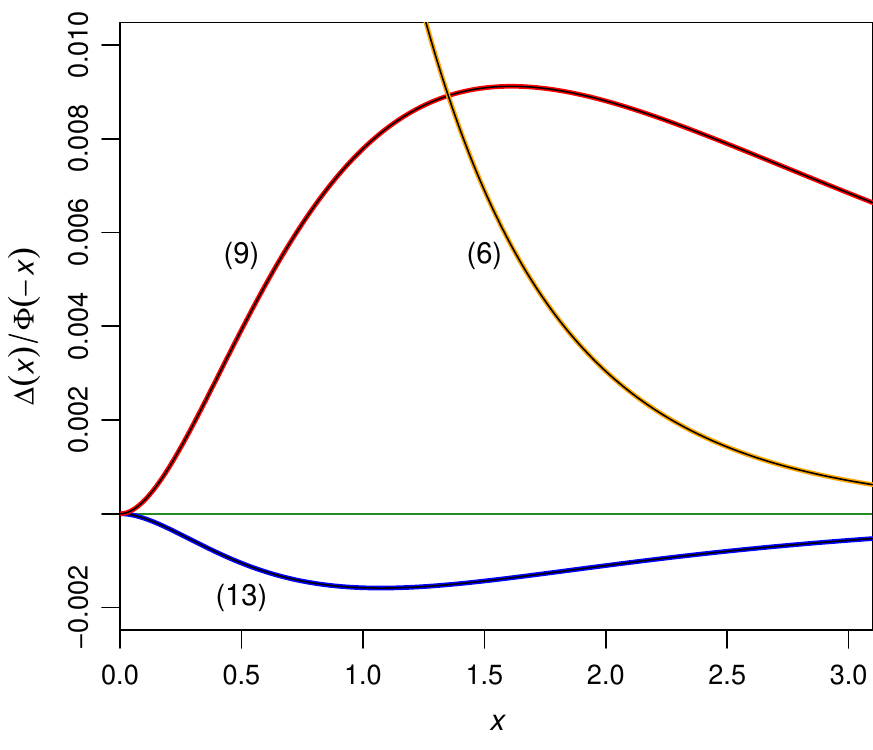}
\caption{Approximation errors $\Delta$ (left panel) and relative errors $\Delta(x)/\Phi(-x)$ (right panel) for some first bounds.}
\label{fig:B1}
\end{figure}

\section{Continued Fractions and General Bounds}
\label{Continued Fractions}

Recall that we started with an arbitrary function $h(x)$, then turned our attention to $h(x) = x + 1/g(x)$ with smooth $g$, and the special functions $g$ we used may also be written as $g(x) = x + 1/\tilde{g}(x)$ with another smooth function $\tilde{g} > 0$. After playing around with the resulting approximation error $\Delta = \phi/h - (1 - \Phi)$ and $\sign(\Delta')$, the following scheme seems to be promising:
\[
	h(x) \ = \ g_0(x) \ = \ x + \frac{1}{g_1(x)}
	\ = \ x + \frac{1}{x + \displaystyle \frac{2}{g_2(x)}}
\]
for some differentiable function $g_2 : (0,\infty) \to (0,\infty)$. Indeed, elementary calculations reveal that
\begin{align*}
	\sign \bigl( \Delta'(x) \bigr) \
	&= \ \quad \sign \bigl( g_0(x)^2 - x g_0(x) - 0 - g_0'(x) \bigr) \\
	&= \ - \sign \bigl( g_1(x)^2 - x g_1(x) - 1 - g_1'(x) \bigr) \\
	&= \ \quad \sign \bigl( g_2(x)^2 - x g_2(x) - 2 - g_2'(x) \bigr) ,
\end{align*}
and this suggests a more general result which will be proved via induction:

\begin{Lemma}
\label{lem:continued fractions}
Let $g_0, g_1, g_2, \ldots$ be differentiable functions from $(0,\infty)$ or $[0,\infty)$ to $(0,\infty)$, and define
$h_0(x) := g_0(x)$, $h_1(x) := x + 1/g_1(x)$, and for integers $k \ge 2$,
and, for integers $k > 2$,
\[
	h_k(x) \ := \ x + \frac{1}{x + 
				\displaystyle\frac{2}{\begin{array}{c} \, \ddots\\\strut\end{array} x +
						\displaystyle\frac{k}{g_k(x)}}} .
\]
Then for $k \ge 0$, the value $h_k(x)$ is continuous and strictly monotone in $g_k(x)$, increasing if $k$ is even, decreasing if $k$ is odd. The approximation errors $\Delta_k(x) := \phi(x)/h_k(x) - \Phi(-x)$ satisfy
\[
	\sign(\Delta_k'(x)) \ = \ (-1)^k \sign(G_k(x))
	\quad\text{with}\quad
	G_k(x) \ := \ g_k(x)^2 - x g_k(x) - k - g_k'(x) .
\]
\end{Lemma}

\begin{proof}[\bf Proof of Lemma~\ref{lem:continued fractions}]
We know that the claims are correct for $k = 0, 1$. Now suppose that they are correct for an arbitrary integer $k \ge 1$, and let $g_k(x) = x + (k+1)/g_{k+1}(x)$ for a differentiable function $g_{k+1}$ from $(0,\infty)$ or $[0,\infty)$ to $(0,\infty)$. Then $h_{k+1}$ is continuous and strictly monotone in $g_{k+1}$, decreasing if $k$ is even, increasing if $k$ is odd. Furthermore,
\begin{align*}
	G_k(x) \
	&= \ g_k(x) (g_k(x) - x) - k - g_k'(x) \\
	&= \ \Bigl( x + \frac{k+1}{g_{k+1}(x)} \Bigr) \frac{k+1}{g_{k+1}(x)}
		- k - 1 + \frac{(k+1) g_{k+1}'(x)}{g_{k+1}(x)^2} \\
	&= \ - \, \frac{k+1}{g_{k+1}(x)^2} \,
		\bigl( g_{k+1}(x)^2 - x g_{k+1}(x) - (k+1) - g_{k+1}'(x) \bigr) \\
	&= \ - \, \frac{k+1}{g_{k+1}(x)^2} \, G_{k+1}(x) .
\end{align*}
In particular, $\sign \bigl( \Delta'(x) \bigr) = (-1)^k \sign(G_k(x)) = (-1)^{k+1} \sign(G_{k+1}(x))$.
\end{proof}

\paragraph{Special constants $c_k^*$.}
Suppose that the functions $g_k$ in Lemma~\ref{lem:continued fractions} are defined on $[0,\infty)$. Then $(-1)^k \Delta_k(0)$ is a continuous and strictly decreasing function of $g_k(0)$. It is equal to $0$ if and only if $g_k(0) = \sqrt{c_k^*}$, where the constants $c_k^*$ are defined as follows: $c_0^* = 2/\pi$, and for $k \ge 1$,
\begin{equation}
\label{recursion.ck.1}
	c_k^* \ = \ k^2 / c_{k-1}^* .
\end{equation}
In particular, $c_1^* = \pi/2$. Indeed, replacing $h_{k-1}(x)$ with $h_k(x)$ means replacing $g_{k-1}(x)$ with $x + k/g_k(x)$, so $g_{k-1}(0) = \sqrt{c_{k-1}^*}$ if and only if $g_k(0) = \sqrt{k^2/c_{k-1}^*}$. One can also write
\begin{equation}
\label{recursion.ck.2}
	c_k^* \ = \ \Bigl( \frac{k}{k-1} \Bigr)^2 c_{k-2}^*
\end{equation}
for $k \ge 2$. For general values of $g_k(0) > 0$, one can say that $(-1)^k \Delta_k(0)$ is strictly positive or negative if and only if $g_k(0)$ is strictly smaller or strictly larger than $\sqrt{c_k^*}$, respectively.

By means of Lemma~\ref{lem:continued fractions} and these constants $c_k^*$, one can extend our recipes from Section~\ref{First Steps} to get bounds for the approximation errors $\Delta_k(x)$.

\paragraph{\sl Recipe~1($g_k$).} If $G_k > 0$ on $(0,\infty)$, then $(-1)^k \Delta_k < 0$ on $[0,\infty)$. Analogously, if $G_k < 0$ on $(0,\infty)$, then $(-1)^k \Delta_k > 0$ on $[0,\infty)$.

\paragraph{\sl Recipe~2($g_k$).} Suppose that $g_k$ is defined and differentiable on $[0,\infty)$ with $g_k(0) = \sqrt{c_k^*}$, so $\Delta_k(0) = 0$.\\[0.25ex]
Suppose that for some $x_k > 0$, $G_k > 0$ on $(0,x_k)$ and $G_k < 0$ on $(x_k,\infty)$. Then $(-1)^k \Delta_k > 0$ on $(0,\infty)$.\\[0.25ex]
Analogously, if for some $x_k > 0$, $G_k < 0$ on $(0,x_k)$ and $G_k > 0$ on $(x_k,\infty)$, then $(-1)^k \Delta_k < 0$ on $(0,\infty)$.

\subsection{Approximations with rational functions $h$}

\paragraph{Laplace's approximation scheme.}
Note first that $g_k(x) := x$ yields $G_k(x) = - (k+1)$, so by Recipe~1($g_k$), the corresponding approximation error $\Delta_k$ satisfies $(-1)^k \Delta_k > 0$. In other words, $\phi(x)/h_k(x)$ is strictly larger than $\Phi(-x)$ for even $k$ and strictly smaller for odd $k$. Note also that to get from $h_k(x)$ to $h_{k+2}(x)$, the right-most $x$ in $h_k(x)$ is replaced with the strictly larger term $x + (k+1)/(x + (k+2)/x)$, so $h_k(x) < h_{k+2}(x)$ if $k$ is even and $h_k(x) > h_{k+2}(x)$ if $k$ is odd. These considerations prove \eqref{finite continued fractions}.

\paragraph{Rational schemes for $[0,\infty)$.}
Suppose first that $g_k(x) = \sqrt{c_k}$ for some constant $c_k > 0$. Then $G_k(x) = c_k - \sqrt{c_k}\,x - k$. If we take $c_k = k$, then $G_k < 0$ on $(0,\infty)$, so $(-1)^k \Delta_k > 0$ by Recipe~1($g_k$). In particular, $(-1)^k \Delta_k(0) > 0$, and this implies that $\sqrt{k} < \sqrt{c_k^*}$.\\[0.25ex]
Next, let $g_k(x) = \sqrt{c_k} + x$ for some $c_k > 0$. Then, $G_k(x) = c_k - (k+1) + \sqrt{c_k} \, x$. If we take $c_k = k+1$, then $G_k > 0$ on $(0,\infty)$, so $(-1)^k \Delta_k < 0$ by Recipe~1($g_k$). In particular, $(-1)^k \Delta_k(0) < 0$, and this implies that $\sqrt{k+1} > \sqrt{c_k^*}$.\\[0.25ex]
These considerations show that
\[
	k \ < \ c_k^* \ < \ k + 1 .
\]
Starting from these two bounds, one can deduce much more accurate bounds for the constants $c_k^*$.

\begin{Lemma}
\label{bounding ck}
The special constants $c_k^*$ satisfy
\[
	\frac{1}{8(k+1)}
	\ < \ c_k^* - k - 1/2
	\ < \ \frac{1}{8(k + 1/2)}
	\quad \text{for all} \ k \ge 0 .
\]
\end{Lemma}

With these bounds we can prove the following results:

\begin{Lemma}
\label{lem:Laplace.k}
Let $h_k$ be defined as in Lemma~\ref{lem:continued fractions} with $g_k(x) := \sqrt{c_k^*} + \lambda_k x$, $\lambda_k > 0$. This implies that $\Delta_k(0) = 0$, while on $(0,\infty)$,
\[
	(-1)^k \Delta_k \ \begin{cases}
		< \ 0 & \text{if} \ \lambda_k = 1 , \\
		> \ 0 & \text{if} \ \lambda_k = c_k^* - k .
	\end{cases}
\]
\end{Lemma}

\begin{proof}[\bf Proof of Lemma~\ref{bounding ck}]
As shown already, $k < c_k^* < k+1$. Now we compare $c_k^*$ with $d_k := k + 1/2 + \bigl( 8(k + \gamma) \bigr)^{-1}$ for some fixed $\gamma > 0$. Note that $\lim_{k\to\infty} c_k^*/k = 1$ and $\lim_{k\to\infty} d_k/k = 1$, whence
\[
	\lim_{k \to \infty} \, \frac{c_k^*}{d_k} \ = \ 1 .
\]
Suppose we can show that
\begin{equation}
\label{ck vs dk 1}
	\frac{c_{k+2}^*}{d_{k+2}} \ > \ \frac{c_k^*}{d_k}
	\quad\text{for all} \ k \ge 0 .
\end{equation}
Then, for arbitrary integers $k \ge 0$, the sequence $(c_{k+2\ell}/d_{k+2\ell})_{\ell=0}^\infty$ is strictly increasing with limit $1$, whence $c_k^* < d_k$. Analogously, if
\begin{equation}
\label{ck vs dk 2}
	\frac{c_{k+2}^*}{d_{k+2}} \ < \ \frac{c_k^*}{d_k}
	\quad\text{for all} \ k \ge 0 ,
\end{equation}
then $c_k^* > d_k$ for each $k \ge 0$.

Writing $x \sim y$ if $x = \kappa y$ with $\kappa > 0$, we may deduce from the recursive identity \eqref{recursion.ck.2} for the $c_k^*$ that
\begin{align*}
	\frac{c_{k+2}^*}{d_{k+2}} - \frac{c_k^*}{d_k}
	&\sim \ (k+2)^2 d_k - (k+1)^2 d_{k+2} \\
	&= \ (k+2)^2 d_k - (k+1)^2 d_{k+2} \\
	&= \ (k+2)^2(k + 1/2) - (k+1)^2 (k + 5/2)
		+ \frac{(k+2)^2}{8(k+\gamma)} - \frac{(k+1)^2}{8(k+2+\gamma)} \\
	&= \ - \frac{1}{2} + \frac{(k+2)^2}{8(k+\gamma)} - \frac{(k+1)^2}{8(k+2+\gamma)} \\
	&\sim \ (k+2)^2 (k+2+\gamma) - 4(k+\gamma)(k+2+\gamma) - (k+1)^2(k+\gamma) \\
	&= \ (3 - 6\gamma) k + 8 - 5\gamma - 4\gamma^2 .
\end{align*}
In case of $\gamma = 1/2$, the previous expression equals $9/2$, so \eqref{ck vs dk 1} is satisfied.
In case of $\gamma = 1$ we get $- 3k - 1$, so \eqref{ck vs dk 2} holds true.
\end{proof}

\begin{proof}[\bf Proof of Lemma~\ref{lem:Laplace.k}]
Let $g_k(x) := \sqrt{c_k^*} + \lambda_k x$ for some $\lambda_k > 0$. Then $g_k(0) = \sqrt{c_k^*}$ and
\[
	G_k(x) \ = \ c_k^* - k - \lambda_k + (2 \lambda_k - 1) \sqrt{c_k^*} \, x - \lambda_k(1 - \lambda_k) x^2 .
\]
If we set $\lambda_k = 1$, then $G_k(x) = c_k^* - k - 1 + \sqrt{c_k^*} \, x$ is strictly increasing in $x \ge 0$ with $G(0) = c_k^* - k - 1 < 0$ and limit $G_k(\infty) = \infty$. Thus we may deduce from Recipe~2($g_k$) that $(-1)^k \Delta_k < 0$ on $(0,\infty)$.

If we set $\lambda_k = c_k^* - k$, then $\lambda_k \in (0,1)$ because of $k < c_k^* < k+1$, and $\lambda_k > 1/2$ by the lower bound in Lemma~\ref{bounding ck}. Consequently, $G_k(x) = \lambda_k(1 - \lambda_k) x (x_k - x)$ with $x_k := (2\lambda_k - 1) \sqrt{c_k^*}/[\lambda_k(1 - \lambda_k)] > 0$. Thus we may deduce from Recipe~2($g_k$) that $(-1)^k \Delta_k > 0$ on $(0,\infty)$.
\end{proof}

\subsection{Approximations with non-rational functions $h$}

\paragraph{Shenton's and Kouba's approximation scheme.}
For $k \ge 0$ and $j \in \{1,2\}$ consider $g_k(x) = g_{k,j}(x) = \sqrt{k + j/2 + (x/2)^2} + x/2$. The resulting auxiliary function $G_k = G_{k,j}$ equals
\[
	G_{k,j}(x) \ = \ (j - 1)/2 - \frac{x}{2\sqrt{4(k + j/2) + x^2}} .
\]
This is strictly decreasing in $x \ge 0$ with $G_{k,j}(0) = (j - 1)/2$ and $\lim_{x\to \infty} G_{j,k}(x) = j/2 - 1$. Consequently, $G_{k,1} < 0 < G_{k,2}$ on $(0,\infty)$. By means of Recipe~1($g_k$), we may conclude that the approximation error $\Delta_{k,j}(x) = \phi(x)/h_{k,j}(x) - \Phi(-x)$ satisfies
\[
	(-1)^{k+j+1} \Delta_{k,j} \ > \ 0
\]
on $(0,\infty)$, and this proves \eqref{Shenton}.

\paragraph{First refinement.}
If we set $g_k(x) := \sqrt{c_k^* + (x/2)^2} + x/2$, then $g_k(0) = \sqrt{c_k^*}$, and
\[
	G_k(x) \ = \ c_k^* - k - 1/2 - \frac{x}{2 \sqrt{4c_k^* + x^2}}
\]
is strictly decreasing in $x \ge 0$ with $G_k(0) = c_k^* - k - 1/2 > 0$ and $\lim_{x\to\infty} G_k(x) = c_k^* - k - 1 < 0$. Consequently, $(-1)^k \Delta_k$ is strictly positive on $(0,\infty)$ by Recipe~2($g_k$). This refinement, replacing $h_{k,1}$ with $h_k$, is a generalization of Pollak's refinement \eqref{Pollak} of Komatu's upper bound in \eqref{Komatu}.

\begin{Lemma}
\label{lem:Shenton.Pollak}
Let $h_k(x)$ be defined as in Lemma~\ref{lem:continued fractions} with $g_k(x) := \sqrt{c_k^* + (x/2)^2} + x/2$. Then for $x > 0$, the approximation errors $\Delta_k(x) = \phi(x)/h_k(x) - \Phi(-x)$ and the errors $\Delta_{k,j}(x)$ introduced before satisfy
\begin{align*}
	\Delta_{k,1}(x) \, > \, \Delta_k(x) \, > \, 0 \, > \, \Delta_{k,2}(x)
		& \quad\text{if $k$ is even} , \\
	\Delta_{k,1}(x) \, < \, \Delta_k(x) \, < \, 0 \, < \, \Delta_{k,2}(x)
		& \quad\text{if $k$ is odd} .
\end{align*}
\end{Lemma}

Figure~\ref{fig:B2} illustrates these bounds for $k = 0,1,2,3$.

\begin{figure}[h]
\includegraphics[width=0.49\textwidth]{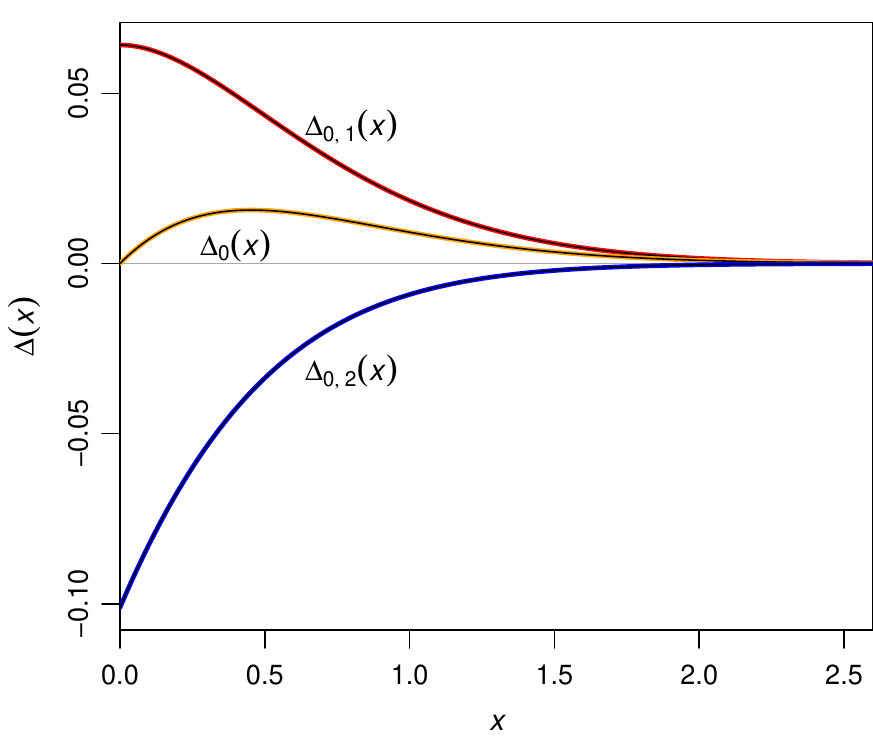}
\hfill
\includegraphics[width=0.49\textwidth]{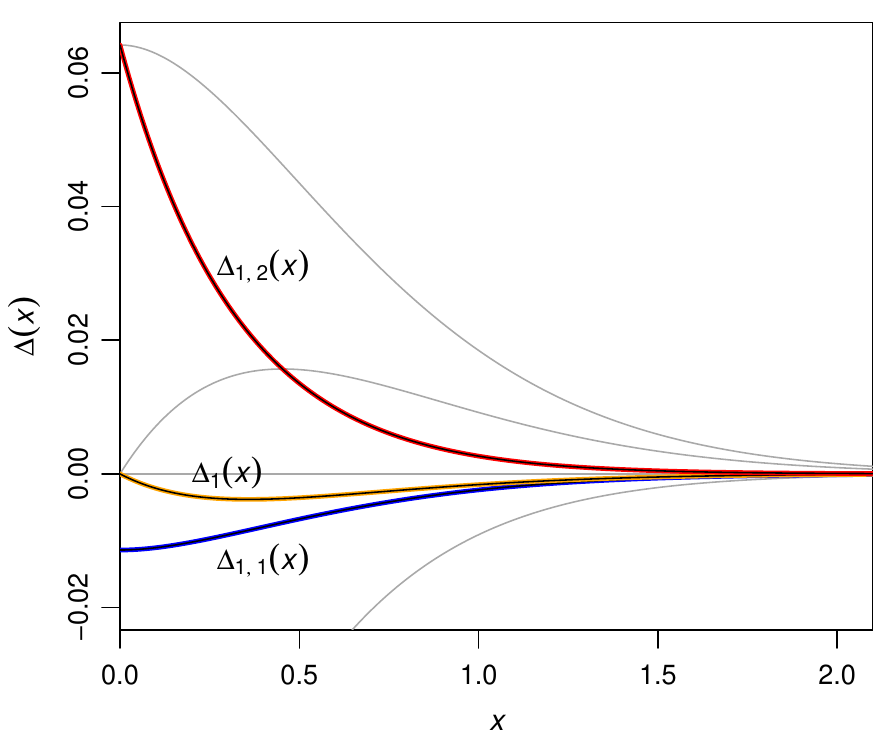}

\includegraphics[width=0.49\textwidth]{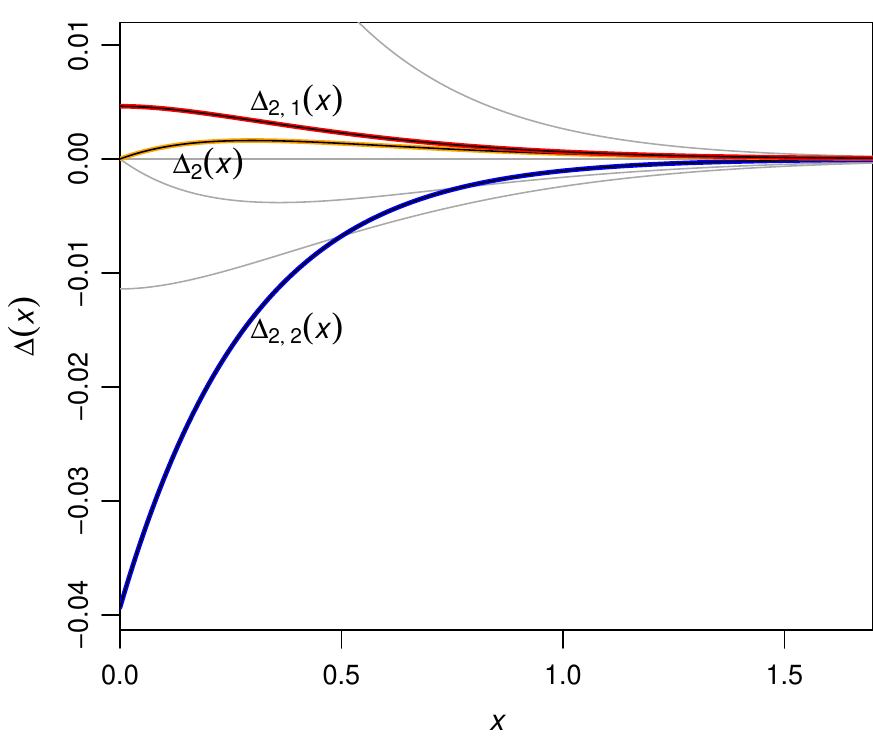}
\hfill
\includegraphics[width=0.49\textwidth]{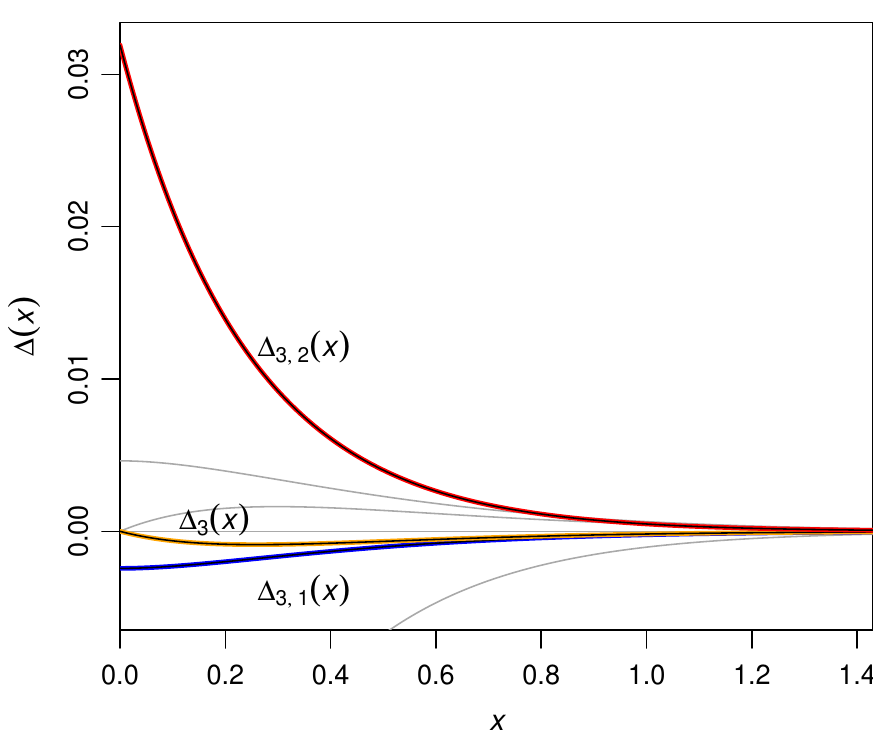}

\caption{Approximation errors $\Delta_{k,1}$, $\Delta_{k,2}$ and $\Delta_k$ for $k = 0, 1, 2, 3$.}
\label{fig:B2}
\end{figure}

\paragraph{Second refinement.}
Similarly as in Section~\ref{First Steps}, we can further improve the bounds in Lemma~\ref{lem:Shenton.Pollak} by considering $\wtilde{g}_k(x) = \sqrt{c_k^* + \tau_k^2x^2} + (1 - \tau_k)x$ with suitable $\tau_k \in (0,1/2)$. That $\wtilde{g}_k(x) > g_k(x)$ follows from \eqref{monotonicity in tau}.

\begin{Theorem}
\label{thm:dabigone}
Let $\wtilde{h}_k(x)$ be defined as $h_k(x)$ in Lemma~\ref{lem:continued fractions} with $g_k(x)$ being replaced by
\[
	\wtilde{g}_k(x) \ := \ \sqrt{c_k^* + (k+1 - c_k^*)^2x^2} + (c_k^* - k) x .
\]
Then for $x > 0$, the approximation errors $\wtilde{\Delta}_k(x) = \phi(x)/\wtilde{h}_k(x) - \Phi(-x)$ satisfy
\begin{align*}
	\Delta_{k,1}(x) \ > \ \Delta_k(x) \ > \ \wtilde{\Delta}_k(x) \ > \ 0 \ > \ \Delta_{k,2}(x)
		& \quad\text{if $k$ is even} , \\
	\Delta_{k,1}(x) \ < \ \Delta_k(x) \ < \ \wtilde{\Delta}_k(x) \ < \ 0 \ < \ \Delta_{k,2}(x)
		& \quad\text{if $k$ is odd} .
\end{align*}
Moreover,
\begin{align*}
	&\wtilde{\Delta}_0(x) \, > \, \wtilde{\Delta}_2(x) \, > \, \wtilde{\Delta}_4(x) \, > \, \cdots , \\
	&\wtilde{\Delta}_1(x) \, < \, \wtilde{\Delta}_3(x) \, < \, \wtilde{\Delta}_5(x) \, < \, \cdots .
\end{align*}
Finally, for all $k \ge 0$ and $x \ge 0$,
\[
	\frac{|\wtilde{\Delta}_k(x)|}{\Phi(-x)} \
	\le \ \frac{1}{8(k + 3/2)^{3/2} \sqrt{k + 3/2 + x^2}} .
\]
\end{Theorem}

Numerical experiments show that the previous upper bound for the relative approximation error $|\wtilde{\Delta}_k(x)| / \Phi(- x)$ is rather conservative. The main point is the explicit conclusion that
\[
	\frac{\wtilde{\Delta}_k(x)}{\Phi(- x)} \ \to \ 0
	\quad\text{as} \begin{cases}
		k \to \infty , & \text{uniformly in} \ x \ge 0 , \\
		x \to \infty , & \text{uniformly in} \ k \ge 0 .
	\end{cases}
\]

\paragraph{Illustration.}
Note that the functions $\wtilde{g}_k(x) = \sqrt{c_k^* + (1/2 - \delta_k)^2 x^2} + (1/2 + \delta_k)x$ and $g_k(x) = \sqrt{c_k^2 + (x/2)^2} + x/2$ are very similar with $\delta_k = c_k^* - k - 1/2$ lying between $1/[8(k + 1)]$ and $1/[8(k + 1/2)]$. Nevertheless, the resulting approximation errors $(-1)^k \wtilde{\Delta}_k$ are substantially smaller than $(-1)^k \Delta_k$. Figure~\ref{fig:B3} depicts the relative approximation errors $\Delta_k/\Phi$ and $\wtilde{\Delta}_k/\Phi$ for $k = 0,1,\ldots,7$.

\begin{figure}[h]
\includegraphics[width=0.49\textwidth]{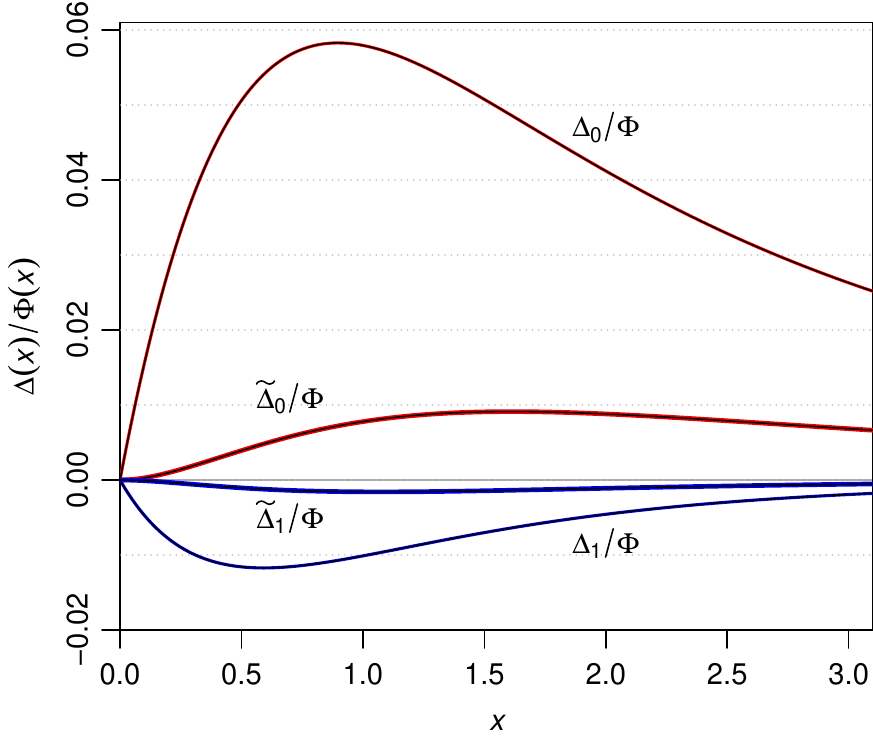}
\hfill
\includegraphics[width=0.49\textwidth]{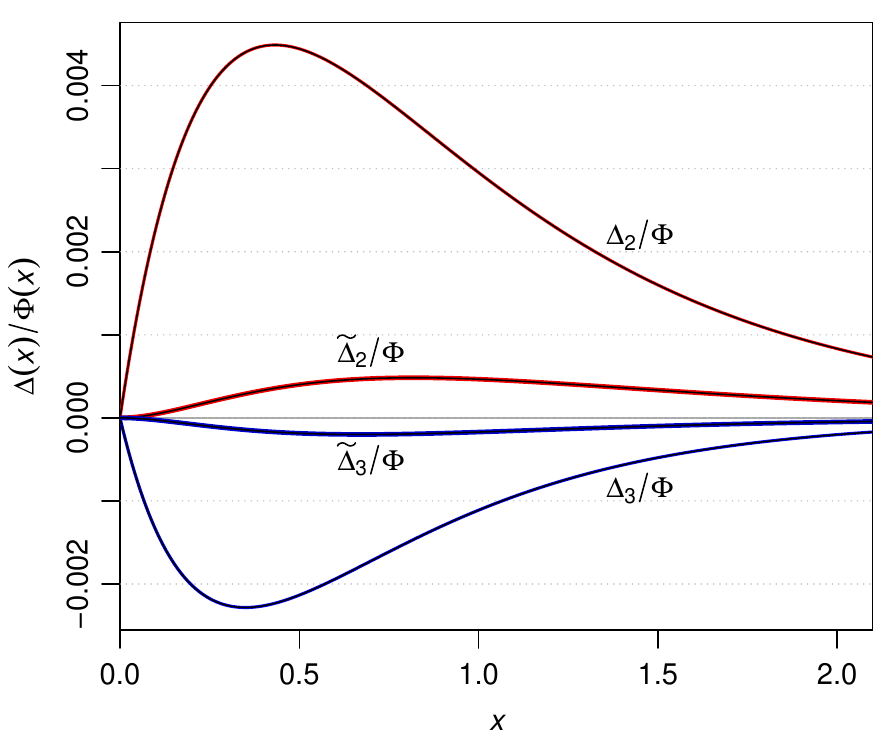}

\includegraphics[width=0.49\textwidth]{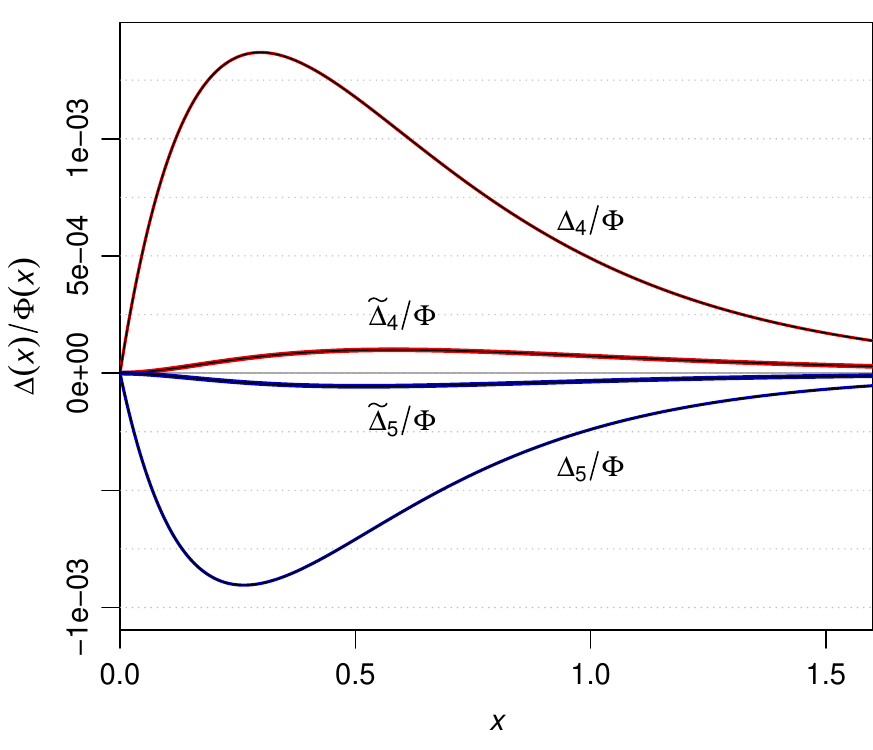}
\hfill
\includegraphics[width=0.49\textwidth]{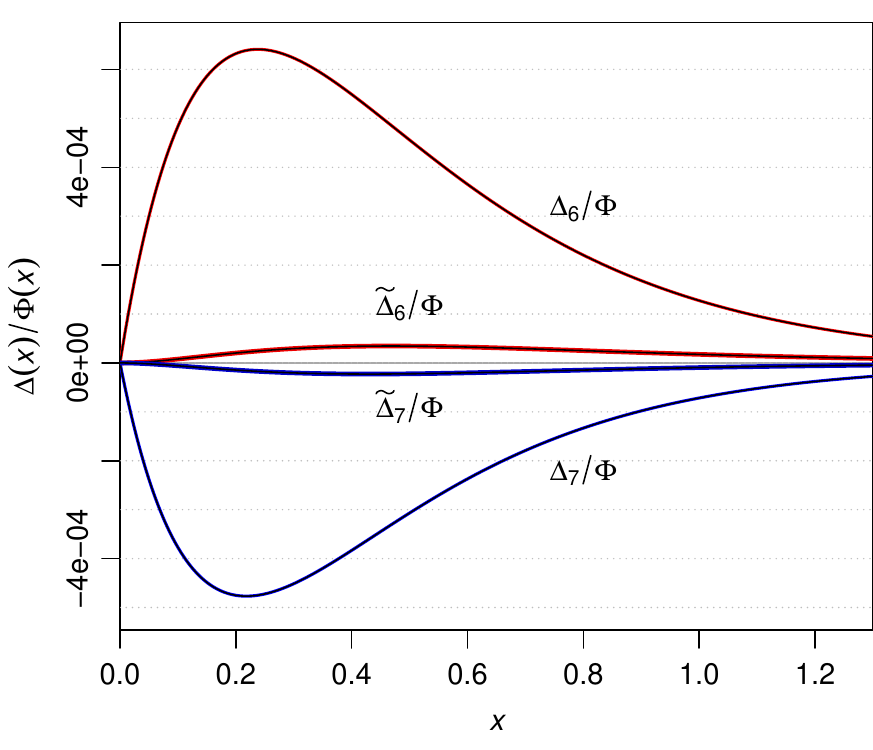}

\caption{Relative approximation errors $\Delta_k/\Phi$ and $\wtilde{\Delta}_k/\Phi$ for $k = 0, 1, \ldots, 7$.}
\label{fig:B3}
\end{figure}

\begin{proof}[\bf Proof of Theorem~\ref{thm:dabigone}]
For $\tau_k \in (0,1)$, let $\wtilde{g}_k(x) := g_k^o(x) + (1 - \tau_k) x$ with $g_k^o(x) := \sqrt{c_k^* + \tau_k^2 x^2}$. Elementary calculations show that $\wtilde{G}_k(x) = \wtilde{g}_k(x)^2 - x \wtilde{g}_k(x) - k - \wtilde{g}_k'(x)$ is given by
\[
	\wtilde{G}_k(x) \
	= \ c_k^* - k - 1 + \tau_k
		+ \frac{(1 - 2\tau_k) c_k^* x}{g_k^o(x) + \tau_k x}
		- \frac{\tau_k^2 x}{g_k^o(x)} .
\]
This motivates the choice $\tau_k := k+1 - c_k^* < 1/2$, because then $\wtilde{G}_k(0) = 0$, so $\wtilde{\Delta}_k(0) = 0 = \wtilde{\Delta}_k^(0)$. Now, further calculations lead to
\[
	\wtilde{G}_k(x) \
	= \ \frac{\bigl( c_k^*(c_k^* + 1) - (k+1)^2 \bigr) x^2}{(g_k^o(x) + \tau_k x) g_k^o(x)}
		\Bigl( \sqrt{\frac{c_k^*}{x^2} + \tau_k^2} - \frac{\tau_k^3}{c_k^*(c_k^* + 1) - (k+1)^2} \Bigr)
\]
for $x > 0$. By means of Lemma~\ref{bounding ck} one can show that $c_k^*(c_k^* + 1) - (k+1)^2 > 0$. Indeed,
\begin{align*}
	c_k^*(c_k^* + 1) \
	&> \ \Bigl( k + 1/2 + \frac{1}{8(k+1)} \Bigr) \Bigl( k + 3/2 + \frac{1}{8(k+1)} \Bigr) \\
	&= \ k^2 + 2k + 3/4 + \frac{2k + 2}{8(k+1)} + \frac{1}{8^2 (k+1)^2} \\
	&= \ (k+1)^2 + \frac{1}{8^2(k+1)^2} .
\end{align*}
Consequently, for $x > 0$, the sign of $\wtilde{G}_k(x)$ is equal to the sign of
\[
	G_k^o(x) \ := \ \sqrt{\frac{c_k^*}{x^2} + \tau_k^2} - \frac{\tau_k^3}{c_k^*(c_k^* + 1) - (k+1)^2} .
\]
But $G_k^o(x)$ is strictly decreasing in $x > 0$ with limit $G_k^o(0) = \infty$. Moreover, $G_k^o(\infty) < 0$, because otherwise $\wtilde{G}_k > 0$ on $(0,\infty)$, which would imply that $\wtilde{\Delta}_k$ is strictly monotone on $(0,\infty)$, a contradiction to $\wtilde{\Delta}_k(0) = 0 = \wtilde{\Delta}_k(\infty)$. Thus we may apply Recipe~2($g_k$) to conclude that $(-1)^k \wtilde{\Delta}_k > 0$ on $(0,\infty)$.

It remains to prove that for $x > 0$, $\wtilde{\Delta}_k(x) > \wtilde{\Delta}_{k+2}(x)$ for even $k$ and $\wtilde{\Delta}_k(x) < \wtilde{\Delta}_{k+2}(x)$ for odd $k$. This is equivalent to the inequality
\[
	\wtilde{g}_k(x) \ < \ x + \frac{k+1}{x + (k+2)/\wtilde{g}_{k+2}(x)} .
\]
But with $\tau_\ell = \ell+1-c_\ell^*$,
\[
	\wtilde{g}_k(x) \
	= \ x + \sqrt{c_k^* + \tau_k^2 x^2} - \tau_k x \
	= \ x + \frac{c_k^*}{\sqrt{c_k^* + \tau_k^2 x^2} + \tau_k x} ,
\]
so the asserted inequality is equivalent to
\[
	\frac{\sqrt{c_k^* + \tau_k^2 x^2} + \tau_k x}{c_k^*/(k+1)} \
	> \ x + \frac{k+2}{\wtilde{g}_{k+2}(x)} .
\]
Since $(k+1)^2/c_k^* = c_{k+1}^*$ and $(k+1) \tau_k/c_k^* = c_{k+1}^* - k - 1 = 1 - \tau_{k+1}$, we have to show that for fixed $\ell \ge 1$ (i.e.\ $\ell = k+1$),
\[
	a_\ell(x) \
	:= \ \sqrt{c_\ell^* + (1 - \tau_\ell)^2 x^2} - \tau_\ell x
	\quad\text{and}\quad
	\wtilde{g}_{\ell+1}(x) \ = \ \sqrt{c_{\ell+1}^* + \tau_{\ell+1}^2 x^2} + (1 - \tau_{\ell+1}) x
\]
satisfy
\begin{equation}
\label{???}
	a_\ell(x) \wtilde{g}_{\ell+1}(x) \ > \ \ell + 1
	\quad\text{for} \ x > 0 .
\end{equation}
The left-hand side is equal to
\begin{align*}
	a_\ell(x) \wtilde{g}_{\ell+1}(x) \
	= \ &\sqrt{(c_\ell^* + (1 - \tau_\ell)^2 x^2)(c_{\ell+1}^* + \tau_{\ell+1}^2 x^2)}
		- \tau_\ell (1 - \tau_{\ell+1}) x^2 \\
		&+ \ \Bigl( (1 - \tau_{\ell+1}) \sqrt{c_\ell^* + (1 - \tau_\ell)^2 x^2}
				- \tau_\ell \sqrt{c_{\ell+1}^* + \tau_{\ell+1}^2 x^2} \Bigr) x ,
\end{align*}
so it suffices to show that for $x > 0$,
\begin{equation}
\label{???1}
	\sqrt{(c_\ell^* + (1 - \tau_\ell)^2 x^2)(c_{\ell+1}^* + \tau_{\ell+1}^2 x^2)} \
	> \ \ell + 1 + \tau_\ell(1 - \tau_{\ell+1}) x^2 , \\
\end{equation}
\begin{equation}
\label{???2}
	(1 - \tau_{\ell+1}) \sqrt{c_\ell^* + (1 - \tau_\ell)^2 x^2} \
	> \ \tau_\ell \sqrt{c_{\ell+1}^* + \tau_{\ell+1}^2 x^2} .
\end{equation}

Concerning \eqref{???1}, since $c_\ell^* c_{\ell+1}^* = (\ell + 1)^2$, it suffices to show that
\begin{align}
\label{???1a}
	c_\ell^* \tau_{\ell+1}^2 + c_{\ell+1}^* (1 - \tau_\ell)^2 \
	&> \ 2 (\ell + 1) \tau_\ell (1 - \tau_{\ell+1}) , \\
\label{???1b}
	(1 - \tau_\ell) \tau_{\ell+1} \
	&> \ \tau_\ell(1 - \tau_{\ell+1}) .
\end{align}
With tedious but elementary calculations, using $c_{\ell+1}^* = (\ell+1)^2/c_\ell^*$ and the definitions of $\tau_\ell, \tau_{\ell+1}$, one can show that \eqref{???1a} is equivalent to
\[
	c_\ell^* \ > \ (\ell + 1) \sqrt{\frac{\ell + 1/2}{\ell + 3/2}} \
	= \ (\ell + 1) \sqrt{ 1 - \frac{1}{\ell + 3/2}} .
\]
But according to Lemma~\ref{bounding ck}, the left-hand side is larger than $\ell + 1/2 + 1/[8(\ell + 1)]$, whereas the right-hand side is smaller than
\[
	(\ell + 1) \Bigl( 1 - \frac{1}{2(\ell + 3/2)} - \frac{1}{8(\ell + 3/2)^2} \Bigr) \
	= \ \ell + 1/2 + \frac{1}{8(\ell + 1)} - \frac{1}{32 (\ell + 1)(\ell + 3/2)^2} ,
\]
where we used the standard inequality $\sqrt{1 - x} < 1 - x/2 - x^2/8$ for $x \in (0,1)$ and elementary algebra. Inequality \eqref{???1b} is equivalent to $\tau_{\ell + 1} > \tau_\ell$, and this is equivalent to $c_\ell^*(c_\ell^* + 1) > (\ell + 1)^2$, which is a consequence of Lemma~\ref{bounding ck}.

Concerning \eqref{???2}, elementary calculations show that $c_\ell^* (1 - \tau_{\ell + 1})^2 = c_{\ell+1}^* \tau_\ell^2$, so it remains to show that $(1 - \tau_\ell)(1 - \tau_{\ell+1}) > \tau_\ell \tau_{\ell+1}$. But this is a direct consequence of the fact that $\tau_\ell, \tau_{\ell+1} \in (1,1/2)$.

It remains to prove the asserted upper bound for $|\wtilde{\Delta}(x)|/\Phi(-x)$. Recall that for even $k \ge 0$, $\phi(x)/\wtilde{h}_k(x) \ge \Phi(-x) \ge \phi(x)/\wtilde{h}_{k+1}(x)$, whence
\[
	0 \ \le \ \frac{\wtilde{\Delta}_k(x)}{\Phi(-x)} \ \le \ \frac{\wtilde{h}_{k+1}(x)}{\wtilde{h}_k(x)} - 1 .
\]
And for odd $k \ge 1$, $\phi(x)/\wtilde{h}_k(x) \le \Phi(-x) \le \phi(x)/\wtilde{h}_{k+1}(x)$, so that
\[
	0 \ \ge \ \frac{\wtilde{\Delta}_k(x)}{\Phi(-x)} \ \ge \ \frac{\wtilde{h}_{k+1}(x)}{\wtilde{h}_k(x)} - 1 \
	\ge \ - \Bigl( \frac{\wtilde{h}_k(x)}{\wtilde{h}_{k+1}(x)} - 1 \Bigr) .
\]
Thus it suffices to show that for arbitrary $k \ge 0$,
\[
	\max \Bigl\{ \frac{\wtilde{h}_{k+1}(x)}{\wtilde{h}_k(x)}, \frac{\wtilde{h}_{k+1}(x)}{\wtilde{h}_k(x)} \Bigr\} \
	\le \ 1 + \frac{1}{8 (k + 3/2)^{3/2} \sqrt{k + 3/2 + x^2}} .
\]
Note that
\[
	\wtilde{g}_k(x) \
	= \ x + \frac{c_k^*}{\sqrt{c_k^* + \tau_k^2 x^2} + \tau_k x} \
	= \ x + \frac{k+1}{\check{g}_{k+1}(x)}
\]
with
\[
	\check{g}_{k+1}(x) \
	:= \ \frac{ \sqrt{c_k^* + \tau_k^2 x^2} + \tau_k x}{c_k^*/(k+1)} \
	= \ \sqrt{ c_{k+1}^* + (1 - \tau_{k+1})^2 x^2} + (1 - \tau_{k+1}) x ,
\]
because $(k+1)^2 / c_k^* = c_{k+1}^*$ and $\tau_k (k+1)/c_k^* = 1 - \tau_{k+1}$. Since $\tau_{k+1} \in (0,1/2)$, $\check{g}_{k+1}(x) \ge \wtilde{g}_{k+1}(x)$. Since for arbitrary numbers $a, b_1, b_2 > 0$,
\[
	\max \Bigl\{ \frac{a + b_2}{a + b_1}, \frac{a + b_1}{a + b_2} \Bigr\} \
	\le \ \max \bigl\{ \frac{b_2}{b_1}, \frac{b_1}{b_2} \Bigr\} ,
\]
we may conclude inductively that with $\ell := k+1$,
\[
	\max \Bigl\{ \frac{\wtilde{h}_{k+1}(x)}{\wtilde{h}_k(x)},
			\frac{\wtilde{h}_{k+1}(x)}{\wtilde{h}_k(x)} \Bigr\} \
	\le \ 1 + \frac{x + \ell/\wtilde{g}_\ell(x)}{x + \ell/\check{g}_\ell(x)} \
	= \ 1 + \frac{\check{g}_\ell(x) - \wtilde{g}_\ell(x)}{(x \check{g}_\ell(x) + \ell) \wtilde{g}_\ell(x)} .
\]
It follows from \eqref{monotonicity in tau} that
\[
	\check{g}_\ell(x) \
	\ge \ \wtilde{g}_\ell(x) \
	\ge \ \sqrt{c_\ell^* + x^2} \ > \ x .
\]
On the other hand, Lemma~\ref{bounding ck} implies that $1 - 2 \tau_\ell \le 1 / [4(\ell + 1/2)]$. Consequently,
\[
	\check{g}_\ell(x) - \wtilde{g}_\ell(x) \
	= \ \frac{(1 - 2 \tau_\ell)x^2}
		{\sqrt{c_\ell^* + (1 - \tau_\ell)^2x^2} + \sqrt{c_\ell^* + \tau_\ell^2 x^2}} \
	\le \ \frac{x^2}{8 (\ell + 1/2) \sqrt{c_\ell^*}} ,
\]
and
\[
	(x \check{g}_\ell(x) + \ell) \wtilde{g}_\ell(x) \
	\ge \ x^2 \sqrt{c_\ell^* + x^2} .
\]
This leads to the upper bound
\[
	\max \Bigl\{ \frac{\wtilde{h}_{k+1}(x)}{\wtilde{h}_k(x)},
		\frac{\wtilde{h}_{k+1}(x)}{\wtilde{h}_k(x)} \Bigr\} \
	\le \ 1 + \frac{1}{8 (\ell + 1/2) \sqrt{c_\ell^*} \sqrt{c_\ell^* + x^2}} .
\]
It implies the asserted bound, because $c_\ell^* > \ell + 1/2 = k + 3/2$.
\end{proof}

\paragraph{Remark.}
In an earlier version of this manuscript, we proposed bounds based on the functions $g_k(x) := x + \sqrt{c_k^*} \exp( - \delta_k x)$ with $\delta_k := \tau_k/\sqrt{c_k^*}$. Our new bounds with $\wtilde{g}_k(x) = \sqrt{c_k^* + \tau_k^2 x^2} + (1 - \tau_k) x$ are strictly better for any $x > 0$, because
\[
	\wtilde{g}_k(x) \ = \ x + \frac{\sqrt{c_k^*}}{\sqrt{1 + \delta_k^2x^2} + \delta_k x} \
	> \ x + \sqrt{c_k^*} \exp(- \delta_k x) ,
\]
a consequence of the inequality $\exp(y) > \sqrt{1 + y^2} + y$ for $y > 0$.

\section{Alternative Representations}
\label{Polynomials}

In this section we describe briefly representations of our approximations in terms of simple fractions rather than continued fractions. This material is rather standard, and we refer to Kouba (2006) for deeper connections with Hermite polynomials.

\begin{Lemma}
\label{lem:simple fractions}
Let $h_0, h_1, h_2, \ldots$ be constructed as in Lemma~\ref{lem:continued fractions}. Define polynomials via $P_0(x) := 1$, $P_1(x) := x$, $Q_0(x) := 0$, $Q_1(x) := 1$ and inductively, for $k = 2, 3, 4, \ldots$,
\begin{align*}
	P_k(x) \ &:= \ (k-1) P_{k-2}(x) + x P_{k-1}(x) , \\
	Q_k(x) \ &:= \ (k-1) Q_{k-2}(x) + x Q_{k-1}(x) .
\end{align*}
Then for arbitrary $k \ge 1$ and $x > 0$,
\[
	h_k(x) \ = \ \frac{k P_{k-1}(x) + P_k(x) g_k(x)}{k Q_{k-1}(x) + Q_k(x) g_k(x)} .
\]
In particular, if $g_k(x) = x$, then $h_k(x) = P_{k+1}(x)/Q_{k+1}(x)$.
\end{Lemma}

Table~\ref{tab:polynomials} contains a list of the polynomials $P_k, Q_k$ for $0 \le k \le 8$.
Suppose that $g_k(x) = x + G_k(x)$ for some $G_k : [0,\infty) \to (0,\infty)$, for instance,
\[
	G_k(x) \ := \ \sqrt{c_k^* + \tau_k^2 x^2} - \tau_k x \
	\ = \ \frac{c_k^*}{\sqrt{c_k^* + \tau_k^2 x^2} + \tau_k x}
\]
as in the previous section. Then,
\[
	h_k \ = \ \frac{P_{k+1} + P_k G_k}{Q_{k+1} + Q_k G_k}
	\quad\text{on} \ [0,\infty) .
\]
Specifically,
\begin{eqnarray*}
	h_2(x) & = & \frac{x^3 + 3x + (x^2 + 1) G_2(x)}
		{x^2 + 2 + x G_2(x)} , \\
	h_3(x) & = & \frac{x^4 + 6x^2 + 3 + (x^3 + 3x) G_3(x)}
		{x^3 + 5x + (x^2 + 2) G_3(x)} , \\
	h_4(x) & = & \frac{x^5 + 10x^3 + 15x + (x^4 + 6x^2 + 3) G_4(x)}
		{x^4 + 9x^2 + 8 + (x^3 + 5x) G_4(x)} , \\
	h_5(x) & = & \frac{x^6 + 15x^4 + 45x^2 + 15 + (x^5 + 10x^3 + 15x) G_5(x)}
		{x^5 + 14x^3 + 33x + (x^4 + 9x^2 + 8) G_5(x)} .
\end{eqnarray*}

\begin{table}[h]
\[
	\begin{array}{|c||ccccccccc|cccccccc|}
	\hline
	 k & \multicolumn{9}{|c|}{P_k(x) = \sum_{j}a_j x^j} & \multicolumn{8}{c|}{Q_k(x) = \sum_j b_j x^j} \\
	   & a_0 & a_1 & a_2 & a_3 & a_4 & a_5 & a_6 & a_7 & a_8 
	   & b_0 & b_1 & b_2 & b_3 & b_4 & b_5 & b_6 & b_7 \\
	\hline\hline
	 0 &  1  &     &     &     &     &     &     &     &     
	   &  0  &     &     &     &     &     &     &     \\
	\hline
	 1 &  0  &  1  &     &     &     &     &     &     &     
	   &  1  &     &     &     &     &     &     &     \\
	\hline
	 2 &  1  &  0  &  1  &     &     &     &     &     &     
	   &  0  &  1  &     &     &     &     &     &     \\
	\hline
	 3 &  0  &  3  &  0  &  1  &     &     &     &     &     
	   &  2  &  0  &  1  &     &     &     &     &     \\
	\hline
	 4 &  3  &  0  &  6  &  0  &  1  &     &     &     &     
	   &  0  &  5  &  0  &  1  &     &     &     &     \\
	\hline
	 5 &  0  & 15  &  0  & 10  &  0  &  1  &     &     &     
	   &  8  &  0  &  9  &  0  &  1  &     &     &     \\
	\hline
	 6 & 15  &  0  & 45  &  0  & 15  &  0  &  1  &     &     
	   &  0  & 33  &  0  & 14  &  0  &  1  &     &     \\
	\hline
	 7 &  0  & 105 &  0  & 105 &  0  & 21  &  0  &  1  &     
	   & 48  &  0  & 87  &  0  & 20  &  0  &  1  &     \\
	\hline
	 8 & 105 &  0  & 420 &  0  & 210 &  0  & 28  &  0  &  1  
	   &  0  & 279 &  0  & 185 &  0  & 27  &  0  &  1  \\
	\hline
	\end{array}
\]
\caption{Auxiliary polynomials $P_k, Q_k$ for $0 \le k \le 8$.}
\label{tab:polynomials}
\end{table}

\begin{proof}[\bf Proof of Lemma~\ref{lem:simple fractions}]
Note first that $h_k(x)$ may be written as
\[
	h_k(x) \ = \ \frac{A_k(x) + P_k(x) g_k(x)}{B_k(x) + Q_k(x)g_k(x)}
\]
with certain polynomials $A_k, P_k, B_k, Q_k$. Indeed,
\begin{align*}
	h_0(x) \ &= \ \frac{0 + 1\cdot g_0(x)}{1 + 0\cdot g_0(x)} ,	\quad \text{so}
		\ [A_0(x),P_0(x),B_0(x),Q_0(x)] = [0,1,1,0] , \\
	h_1(x) \ &= \ \frac{1 + x g_1(x)}{0 + 1\cdot g_1(x)} ,	\quad \text{so}
		\ [A_1(x),P_1(x),B_1(x),Q_1(x)] = [1,x,0,1] .
\end{align*}
For $k \ge 2$, let $g_{k-1}(x) := x + k/g_k(x) = (k + x g_k(x))/g_k(x)$. Then
\begin{align*}
	h_k(x)
	= h_{k-1}(x) \
	&= \ \frac{A_{k-1}(x) g_k(x) + P_{k-1}(x) (k + x g_k(x))}
	           {B_{k-1}(x) g_k(x) + Q_{k-1}(x) (k + x g_k(x))} \\
	&= \ \frac{k P_{k-1}(x) + (A_{k-1}(x) + x P_{k-1}(x)) g_k(x)}
	           {k Q_{k-1}(x) + (B_{k-1}(x) + x Q_{k-1}(x)) g_{k}(x)} ,
\end{align*}
i.e.
\begin{align*}
	A_k(x) \ = \ k P_{k-1}(x) \ &\text{and} \ \ P_k(x) \ = \ A_{k-1}(x) + x P_{k-1}(x) , \\
	B_k(x) \ = \ k Q_{k-1}(x) \ &\text{and} \ \ Q_k(x) \ = \ B_{k-1}(x) + x Q_{k-1}(x) .
\end{align*}
Since $A_1(x) = 1\cdot P_0(x)$ and $B_1(x) = 1 \cdot Q_0(x)$, we may write
\[
	h_k(x) \ = \ \frac{k P_{k-1}(x) + P_k(x) g_k(x)}{k Q_{k-1}(x) + Q_k(x)g_k(x)}
	\quad\text{for} \ k \ge 1 ,
\]
where
\begin{align*}
	P_k(x) \ &= \ (k-1) P_{k-2}(x) + x P_{k-1}(x)	\quad\text{and} \\
	Q_k(x) \ &= \ (k-1) Q_{k-2}(x) + x Q_{k-1}(x)	\quad\text{for} \ k=2,3,4, \ldots .
\end{align*}\\[-8ex]
\end{proof}

\paragraph{Acknoledgement.}
I am grateful to Jon A.\ Wellner for stimulating discussions and numerous references.

\end{document}